\documentclass[11pt,a4paper]{amsart}
\usepackage{amssymb,amsmath,epsfig,graphics,mathrsfs}
\usepackage{amssymb,amsmath,epsfig,mathrsfs, enumerate, xparse, mathtools}
\usepackage[pagewise]{lineno}
\usepackage[pagebackref,colorlinks=true,linkcolor=blue,citecolor=blue]{hyperref}
\usepackage{graphicx}

\usepackage{fancyhdr}
\usepackage{hyperref}
\hypersetup{
 colorlinks   = true,
 urlcolor     = blue,
 linkcolor    = blue,
 citecolor   = red ,
 bookmarksopen=true
}

\usepackage{amsmath}
\usepackage{amsfonts}
\usepackage{amssymb}
\usepackage{amsthm}
\usepackage{epsfig,graphics,mathrsfs}
\usepackage{graphicx}

\usepackage{latexsym} 
\usepackage{amssymb,amsmath}
\usepackage{amsthm}
\usepackage{longtable,booktabs,setspace} 
\usepackage{url}

\usepackage[usenames, dvipsnames]{color} 

\usepackage{hyperref}

\newcommand*\MSC[1][1991]{\par\leavevmode\hbox{%
\textit{\,\,\,\,\, #1 Mathematical subject classification:\ }}}
\newcommand\blfootnote[1]{%
  \begingroup
  \renewcommand\thefootnote{}\footnote{#1}%
  \addtocounter{footnote}{-1}%
  \endgroup
}

\def \phi {\varphi}

\def \R {\mathbb{R}}

\def \G{\Gamma}
\newcommand{\Ba}{\mathscr B_x^{(a)}}
\newcommand{\Barr}{\mathscr B_r^{(2\alpha-1)}}

\def \vf{\varphi}

\newcommand{\Rn}{\mathbb R^n}
\newcommand{\Rm}{\mathbb R^m}

\newcommand{\Hn}{\mathbb H^n}

\newcommand{\p}{\partial}

\newcommand{\bG}{\mathbb {G}}
\newcommand{\bg}{\mathfrak g}

\newcommand{\la}{\lambda}

\numberwithin{equation}{section}

\newcommand{\beq}{\begin{equation}}
\newcommand{\bea}[1]{\begin{array}{#1} }
\newcommand{\eeq}{ \end{equation}}
\newcommand{\ea}{ \end{array}}

\newcommand{\nh}{\nabla_H}

\newcommand{\sa}{\langle}
\newcommand{\da}{\rangle}

\newtheorem{theorem}{Theorem}[section]
\newtheorem{lemma}[theorem]{Lemma}
\newtheorem{proposition}[theorem]{Proposition}

\newtheorem{prob}{Problem}

\numberwithin{equation}{section}

\begin{document}

\title[]{Fractal Mehler kernels and nonlinear geometric flows}

{\blfootnote{\MSC[2020]{35K08, 35R11, 53C18, 35K55}}}

\date{}

\begin{abstract}
In this paper we introduce a two-parameter family of Mehler kernels and connect them to a class of Baouendi-Grushin flows in fractal dimension. We also highlight a link with a geometric fully nonlinear equation and formulate two questions.  
\end{abstract}

\dedicatory{``Beyond this there is nothing but prodigies and fictions, the only inhabitants are the poets and inventors of fables; there is no credit, or certainty any farther." \\\emph{Theseus, Plutarch's Lives}}

\author{Nicola Garofalo}
\address{School of Mathematical and Statistical Sciences\\ Arizona State University}\email[Nicola Garofalo]{nicola.garofalo@asu.edu}

\maketitle

\tableofcontents

\section{Introduction}\label{S:intro}

The heat equation in a stratified nilpotent Lie group $\bG$ is the $L^2$ gradient flow
\[
\frac{\p f}{\p t} = - \frac{\p \mathscr E_2(f)}{\p f}
\]
of the energy 
\[
\mathscr E_2(f) = \frac 12 \int_{\bG} |\nh f|^2,
\]
where $|\nh f|^2 = \sum_{i=1}^m (X_i f)^2$ is the left-invariant \emph{carr\'e du champ} associated with a basis of the horizontal (bracket generating) layer of the Lie algebra.

The discovery of explicit heat kernel formulas in nilpotent Lie groups endowed with special symmetries has played a central role in harmonic analysis, geometric analysis, and the theory of subelliptic operators. In their seminal and independent works \cite{Gav, Hu}, Gaveau and Hulanicki established the representation for the heat kernel in any Lie group \(\bG\) of Heisenberg type
\begin{align}\label{GH}
G_2((z,\sigma),t)
   = \frac{2^k}{(4\pi t)^{\frac m2+k}}
     \int_{\R^k} e^{-\frac{i}{t}\langle\sigma,\lambda\rangle}
     \left(\frac{|\lambda|}{\sinh |\lambda|}\right)^{\frac m2}
     e^{- \frac{|z|^2}{4t}\frac{|\lambda|}{\tanh |\lambda|}}
     \, d\lambda.
\end{align}
Cygan subsequently extended the representation \eqref{GH} to the broader class of stratified nilpotent Lie groups of step two; see \cite{Cy}. Throughout this paper, the symbol \(k\) will always refer to the ``vertical" dimension of $\bG$\footnote{If the Lie algebra \(\bg\) of \(\bG\) decomposes as \(\bg = \mathfrak h \oplus \mathfrak v\), with \([\mathfrak h, \mathfrak h] = \mathfrak v\) and \([\mathfrak h,\mathfrak v] = \{0\}\), we denote by \(m=\dim \mathfrak h\) and \(k=\dim \mathfrak v\). Identifying \(\mathfrak h \cong \R^m\) and \(\mathfrak v\cong \R^k\), the pair \((z,\sigma)\) represents the logarithmic coordinates of the group element \(\exp(z,\sigma)\in \bG\).}.

The prototypical example of Lie groups of Heisenberg type is the $2n+1$-dimensional Heisenberg group \(\Hn\), for which $k=1$; however, beyond \(\Hn\), there exists a rich and natural family of such groups. Notably, in the Iwasawa decomposition $\bG = K A N$ of a simple Lie group of real rank one, the nilpotent component $N$ is always a group of Heisenberg type; see \cite{Ka2}, \cite{Adam85}, and \cite{CDKR}. We recall that such Lie groups can be described as the boundary at infinity of the  real, complex, quaternionic or octonionic hyperbolic space, see \cite[Sec. 9.3]{Mo} and also \cite[Sec. 10]{Pansu89}.

The natural family of dilations in \(\bG\), associated with the step two stratification of its Lie algebra, is given by
\[
\delta_\ell(z,\sigma)=(\ell z,\ell^2\sigma),
\]
from which it follows that the homogeneous dimension of the group is \(Q=m+2k\). This is consistent with the structure of \eqref{GH}, which shows that the heat kernel \(G_2\) is homogeneous of degree \(\kappa=-Q\) with respect to the heat dilations
\[
\delta^{(h)}_\ell((z,\sigma),t)=(\delta_\ell(z,\sigma),\ell^2 t),
\]
see also \cite[(3.2) in Theor. 3.1]{Fo75}.

A remarkable feature of the Fourier integral \eqref{GH} is that it \emph{reveals} the intrinsic Kor\'anyi-Folland gauge\footnote{Such function first appeared in \cite{KV}, see also \cite{Adam85, KoR}.}
\begin{equation}\label{N}
N(z,\sigma) = (|z|^4 + 16|\sigma|^2)^{1/4}
\end{equation}
which appears in the explicit fundamental solution of the horizontal Laplacian $\Delta_H = \sum_{j=1}^m X_j^2$ on $\bG$, first obtained by Folland \cite{Fo} in the Heisenberg group $\Hn$ (see also \cite{FScpam}), and subsequently generalized by Kaplan \cite{Ka} to all groups of Heisenberg type. By the word ``reveals" in the above comment we mean the following: in the well-known Euclidean heat kernel $g(x,t) = (4\pi t)^{-\frac n2} e^{-\frac{|x|^2}{4t}}$, the distance function $|x|$ appears explicitly. In the sub-Riemannian heat kernel \eqref{GH} there exists no hint of the function \eqref{N}. Despite this, we have the following result, which is a by-product of the main conformal theorem in the work \cite{GTaim} (see Cor. 1.3):
\begin{equation}\label{meh}
\int_0^\infty G_2((z,\sigma),t)\, dt
= C(m,k)\ N(z,\sigma)^{-(m+2k-2)},
\end{equation}
where
\[ 
C(m,k) = 2^{\frac m2+2k-2} \Gamma(\frac m4) \Gamma(\frac12(\frac m2 + k -1)) \pi^{-\frac{m+k+1}{2}}.
\]
As a consequence of \eqref{meh}, if we had no a priori knowledge of the important function \eqref{N}, then by running the heat flow \eqref{GH}, we would be forced to discover it.

\subsection{Generalized Mehler kernels}
The first objective of this paper is to introduce the following two-parameter extension of \eqref{GH}:
\begin{align}\label{GHgen}
G_{\alpha,\beta}((z,\sigma),t)
 = \frac{2^k}{(4\pi t)^{\beta}}
   \int_{\R^k} e^{-\frac{i}{t}\langle\sigma,\lambda\rangle}
   \left(\frac{|\lambda|}{\sinh |\lambda|}\right)^{\alpha}
   e^{- \frac{|z|^2}{4t}\frac{|\lambda|}{\tanh |\lambda|}}
   \, d\lambda,
   \qquad \beta>\alpha>0.
\end{align}
One should note that the kernel $G_{\alpha,\beta}$ remains spherically symmetric in both $z$ and $\sigma$, and it is homogeneous of degree $\kappa = - 2\beta$ with respect to the above mentioned heat dilations $\delta^{(h)}_\ell((z,\sigma),t)$. Indeed,  one checks immediately from \eqref{GHgen} that
\[
\delta^{(h)}_\ell G_{\alpha,\beta} = \ell^{-2\beta}\, G_{\alpha,\beta}.
\]

In Theorem \ref{T:gengus} and Proposition \ref{P:limit} we will prove that, up to a computable universal constant, \eqref{GHgen} represents the heat kernel, with pole at the origin, of a degenerate parabolic operator on $\R^+_r\times \R^k_\sigma\times \R^+_t$ which arises naturally from a reflected Bessel process of fractal dimension $m_\alpha = 2\alpha$.

More precisely, consider the degenerate evolution operator
\begin{equation}\label{P}
\mathfrak P_{\alpha,k}=\partial_t-\mathscr L_{\alpha,k}
  = \partial_t - \partial_{rr}
    - \frac{2\alpha-1}{r}\partial_r
    - \frac{r^2}{4}\Delta_\sigma,
\end{equation}
acting on $(r,\sigma,t)\in \R^+_r\times \R^k_\sigma\times \R^+_t$. We call this operator a fractal \emph{Baouendi-Grushin} flow since, when $m = 2\alpha\in \mathbb N$, it represents the action of the operator in $\Rm_z\times \R^k_\sigma\times \R^+_t$
\[
\p_t - \Delta_z - \frac{|z|^2}{4}\Delta_\sigma
\]
on functions $u((r,\sigma),t)$, with $r = |z|$.
In \eqref{P} the Bessel operator
\[
\Barr = \partial_{rr} + \frac{2\alpha-1}{r}\partial_r
\]
is classical and plays a central role in analysis, geometry, and stochastic processes with radial symmetries; see \cite{We1, We, We2, MS, Ta, MO, IM, KT, CS, GVjfa, GVacv}. The operator $\Barr + \Delta_\sigma$ is also known as the \emph{Weinstein operator}, but the novelty here is the degenerate factor $\frac{r^2}{4}$ in front of $\Delta_\sigma$. Such factor makes the analysis more delicate, but also gives a special geometric significance to \eqref{P}. As it will be clear from Theorems \ref{T:gengus} and \ref{T:gen}, the broader perspective of this paper is closely connected to the geometric framework in \cite{BOO}, and especially to that in the ``extension" work \cite{FGMT}, see also the subsequent papers \cite{RTaim, MOZjga, MOZ, RT, FOZ, GTjam, GTpotan}. We mention here that, instead of $\Delta_\sigma$, we can also consider the situation of a Baouendi-Grushin flow associated with two fractional Bessel processes intertwined as in 
\[
\p_t - \Barr - \frac{r^2}{4} \left[\p_{ss} + \frac{k-1}s \p_s\right] = \p_t - \Barr - \frac{r^2}{4} \mathscr B_s^{(k-1)},\ \ \ \ \alpha>0, k>0.
\]   
For this class, we have results corresponding to the ones listed below, but for the sake of exposition, we have preferred to stick with \eqref{P} and avoid additional technical aspects.
 
Our first main result identifies the heat kernel for the Cauchy problem
\begin{equation}\label{cp}
\mathfrak P_{\alpha,k} u = 0,
\qquad
u((r,\sigma),0)=\varphi(r,\sigma),
\end{equation}
subject to the reflected condition
\begin{equation}\label{neu}
\lim_{r\to 0^+} r^{2\alpha-1}\partial_r u((r,\sigma),t)=0.
\end{equation}
In what follows, we will denote by 
\[
\Sigma^+= \{\vf \in C^\infty(\R^+\times \R^k)\mid \gamma_{\ell,m,\beta}(\vf)<\infty,\  \forall \ell, m\in \mathbb N_0,\ \forall \beta\in \mathbb N_0^k\},
\]
where we have let
\[
\gamma_{\ell,m,\beta}(\vf) :=  \underset{(r,\sigma)\in \R^+\times\R^k}{\sup} \left|r^\ell \left(\frac 1r \p_r\right)^m\p^\beta_\sigma \vf(r,\sigma)\right| < \infty.
\]
The family $\gamma_{\ell,m,\beta}$ is a countable collection of seminorms which generates a Frechet space topology on $C^\infty(\R^+\times \R^k)$. Note that since $\gamma_{0,1,0}(\vf)<\infty$, we have $|\p_r \vf(r,\sigma)|\le \gamma_{0,1}(\vf) r$ for any $r>0$ and any $\sigma\in \R^k$.  This guarantees in particular that any function $\vf\in \Sigma^+$ satisfies the Neumann condition \eqref{neu} when $\alpha>0$. 

\begin{theorem}\label{T:gengus}
Let $\varphi\in \Sigma^+$. The unique solution of
\eqref{cp}–\eqref{neu} is
\[
u((r,\sigma),t)
   = \int_0^\infty\!\int_{\R^k}
   \mathscr K_{\alpha,k}((r,\sigma),(\rho,\sigma'),t)\,
     \varphi(\rho,\sigma')
     \, d\sigma' \,\rho^{2\alpha-1} d\rho,
\]
where the \emph{generalized Mehler kernel} is
\begin{align}\label{gus10}
\mathscr K_{\alpha,k}((r,\sigma),(\rho,\sigma'),t)
  = \frac{(r\rho)^{1-\alpha}}{\pi^k (2t)^{k+1}}
    \int_{\R^k} e^{-\frac{i}{t}\langle \sigma'-\sigma,\lambda\rangle}
      \frac{|\lambda|}{\sinh |\lambda|}
      e^{-\frac{|\lambda|}{\tanh |\lambda|}
         \frac{r^2+\rho^2}{4t}}
      I_{\alpha-1}\!\left(
        \frac{|\lambda|\rho r}{2t\sinh |\lambda|}
      \right)
      \, d\lambda.
\end{align}
\end{theorem}

In \eqref{gus10} we have indicated with $I_\nu(z)$ the modified Bessel function of order $\nu$. The connection between \eqref{GHgen} and the heat kernel \eqref{gus10} of $\mathfrak P_{\alpha,k}$ is made explicit in the following proposition.

\begin{proposition}\label{P:limit}
For every $(r,\sigma)\in \R^+\times \R^k$,
\begin{equation}\label{limit}
\mathscr K_{\alpha,k}((r,\sigma),(0,0),t)
   = \lim_{(\rho,\sigma')\to (0,0)}
     \mathscr K_{\alpha,k}((r,\sigma),(\rho,\sigma'),t)
   = \frac{2\pi^\alpha}{\Gamma(\alpha)}\, G^\star_{\alpha,\alpha+k}((r,\sigma),t),
\end{equation}
where
\begin{align}\label{gusrad}
G^\star_{\alpha,\beta}((r,\sigma),t)
  := \frac{2^k}{(4\pi t)^\beta}
    \int_{\R^k} e^{-\frac{i}{t}\langle\sigma,\lambda\rangle}
      \left(\frac{|\lambda|}{\sinh |\lambda|}\right)^\alpha
      e^{- \frac{r^2}{4t}\frac{|\lambda|}{\tanh |\lambda|}}
      \, d\lambda,
      \qquad \beta>\alpha>0.
\end{align}
\end{proposition}

Since from \eqref{GHgen}, \eqref{gusrad}, we see that $G_{\alpha,\beta}((z,\sigma),t)=G^\star_{\alpha,\beta}((|z|,\sigma),t)$, the limit relation \eqref{limit} indicates that the parameter choice $\beta=\alpha+k$ is forced by the homogeneity of degree $2$ of the operator $\mathfrak P_{\alpha,k}$. 
In fact, this seemingly natural dimensional restriction conceals a deeper analytic structure, intimately related to the Gegenbauer's identity \eqref{ohoh}, Kummer’s transformation formula \eqref{hyperG} for the Gauss hypergeometric function, and to the Bateman's identity \eqref{hyperGint}. From a geometric perspective, this hidden phenomenon manifests itself in the uniqueness of the parameter $\beta=\alpha+k$, which is the sole value for which the following conformal theorem can possibly hold.

\begin{theorem}\label{T:gen}
For every $\alpha>0$, and with $N(z,\sigma)$ as in \eqref{N}, we have
\begin{equation}\label{gen}
\mathscr E_{\alpha,\alpha+k}(z,\sigma)
   := \int_0^\infty G_{\alpha,\alpha+k}((z,\sigma),t)\, dt
   = \frac{2^{\alpha+2k-4}\Gamma(\frac{\alpha}{2})
      \Gamma(\frac{\alpha+k-1}{2})}
      {\pi^{\frac{2\alpha+k+1}{2}}}
     N(z,\sigma)^{-(2\alpha+2k-2)}.
\end{equation}
\end{theorem}

Comparing \eqref{gen} with \eqref{meh} shows that $m_\alpha =2\alpha$ plays the role of a fractal dimension associated with the variable $r\in \R^+$.

The plan of the paper is as follows. In Section \ref{S:meh} we prove Theorem \ref{T:gengus} and Proposition \ref{P:limit}. Section \ref{S:gen} is devoted to the proof of Theorem \ref{T:gen}. In the final Section \ref{S:np} we highlight an interesting link of our results with the geometric fully nonlinear equation \eqref{fullyH}, and we formulate two questions which we feel are of interest in connection with geometric flows in sub-Riemannian geometry.




\section{Proof of Theorem \ref{T:gengus}}\label{S:meh}

In this section we prove Theorem \ref{T:gengus}. Given $\alpha >0$ and $k\in \mathbb N$, , we consider the operator
\begin{equation}\label{Lablap}
\mathscr L_{\alpha,k} = \p_{rr} + \frac{2\alpha-1}{r} \p_r + \frac{r^2}{4}\Delta_\sigma,
\end{equation}
where $r> 0$ and $\sigma\in \R^k$.
Henceforth, we denote by
\begin{equation}\label{Ba}
\Ba u = \p_{xx} u + \frac a{x} \p_x u = x^{-a} \p_x(x^a \p_x)
\end{equation}
the Bessel operator on the half-line $\R^+$. As it is well-known, with its associated invariant measure $d\omega_a(x) = x^a dx$, such operator
plays a central role in analysis and geometry, particularly in the study of partial differential equations and stochastic processes in which symmetries are involved, see \cite{We1, We, We2, MS, Ta, MO, IM, KT}. Using \eqref{Ba}, we can write \eqref{Lablap} as follows
\begin{equation}\label{Lab2}
\mathscr L_{\alpha,k} = \Barr + \frac{r^2}{4}  \Delta_\sigma.
\end{equation} 
Since \eqref{Lab2} is translation-invariant in $\sigma\in \R^k$, to solve the Cauchy problem
\begin{equation}\label{cpab}
\p_t u - \mathscr L_{\alpha,k} u = 0,\ \ \ \ \ u((r,\sigma),0) = \vf(r,\sigma),
\end{equation}
we apply a partial Fourier transform with respect to this variable,
\[
\hat u((r,\la),t) = \int_{\R^k} e^{-2\pi i\sa\la,\sigma\da} u((r,\sigma),t)d\sigma.
\]
If for any fixed $\la\in\R^k\setminus\{0\}$, we let
\[
v(r,t) = \hat u((r,\la),t),
\]
then \eqref{cpab} is converted into a Cauchy problem for the harmonic oscillator associated with the Bessel process $\Barr$,
\begin{equation}\label{cpak}
\p_t v - \Barr v + \pi^2|\la|^2 r^2 v = 0,\ \ \ \ \ v(r,0) = \hat \vf(r,\la).
\end{equation}  
Since we are interested in reflected Brownian motion, we also impose the Neumann condition \eqref{neu}. The heat kernel for \eqref{cpak} can be obtained using the spectral analysis of the operator $\Barr v - \pi^2|\la|^2 r^2$. This however would lead to a lengthy introduction to the relevant calculus. We instead follow an alternative route, purely PDE based. We will need the following well-known fact; see \cite{OU, BGV, SS, Bo2, Wil}.

\begin{proposition}\label{P:ou} The solution of the Cauchy problem for the Ornstein-Uhlenbeck operator in $\Rm\times (0,\infty)$ 
\begin{equation}\label{cpou}
\begin{cases}
u_t - \Delta u + 2 \omega \langle x,\nabla u\rangle  = 0,\ \ \ \ \ \omega>0,
\\
u(x,0) = \psi(x),
\end{cases}
\end{equation}
is given by the classical Mehler formula
\begin{align}\label{gustavo}
u(x,t) & = (4\pi)^{- \frac m2} e^{m t \omega} \left(\frac{2\omega}{\sinh(2t \omega)}\right)^{\frac m2}
\\
& \times \int_{\Rm} \exp\left( -  \frac{\omega}{2 \sinh(2t \omega)} |e^{t \omega} y - e^{-t \omega} x|^2\right) \psi(y) dy.
\notag
\end{align}
\end{proposition}

To solve \eqref{cpak}, \eqref{neu}, we convert such problem into one for the Ornstein-Uhlenbeck operator by means of the following.

\begin{lemma}\label{L:ouho}
Let $\Phi\in C(\R^{m+1})$ and $h\in C^2(\R^{m+1})$ be connected by the following  nonlinear equation
\begin{equation}\label{riccati}
h_t - \Barr h + (\p_r h)^2  = \Phi.
\end{equation}
Then $v$ solves the partial differential equation
\begin{equation}\label{pdeho}
v_t -\Barr v + \Phi v  = 0
\end{equation}
if and only if $f$ defined by the transformation
\begin{equation}\label{genexp}
v(r,t) = e^{-h(r,t)} f(r,t),
\end{equation} 
solves the equation
\begin{equation}\label{PDEou}
f_t - \Barr f + 2 \p_r h \p_r  f  = 0.
\end{equation}
\end{lemma}

\begin{proof}
From \eqref{genexp} we have
\[
v_t = - f e^{-h} h_t + e^{-h} f_t,
\]
and
\begin{align*}
\Barr v & = f \Barr(e^{-h}) + e^{-h} \Barr f + 2 \p_r (e^{-h}) \p_r f
\\
& = e^{-h} f (\p_r h)^2 - e^{-h} f \Barr h + e^{-h} \Barr f  - 2 e^{-h} \p_r h \p_r f,
\end{align*}
for the function $v$ we have
\[
\p_t v - \Barr v + v[\p_t h - \Barr h + (\p_r h)^2] = e^{-h}[\p_t f - \Barr f + 2 \p_r h \p_r f].
\]
The desired conclusion immediately follows from this identity.

\end{proof}

For the harmonic oscillator PDE
\[
\p_t v - \Barr v + \pi^2|\la|^2 r^2 v = 0,
\]
it is clear that we must take $\Phi(r,t) =  \omega^2 r^2$, with $\omega = \pi |\la|>0$. With this right-hand side,  we try the ansatz $h(r,t) = \frac{A}{2} r^2 + C t$. Such $h$ solves \eqref{riccati} if and only if
$A = \omega$, and  $C = (1+a)\omega$. This gives $h(r,t) = \frac{\omega}{2} r^2 + (1+a)\omega t$. From Lemma \ref{L:ouho} we infer that for every $\la\in \R^k\setminus\{0\}$ the function $f(r,t) = e^{\frac{\pi |\la|}{2} r^2 + (1+a)\pi |\la| t} v(r,t)$ solves the problem
\begin{equation}\label{cpou}
f_t - \Barr f + 2 \omega r \p_r  f  = 0,\ \ \ \ \ f(r,0) = e^{\frac{\omega}{2} r^2} \hat \vf(r,\la),\ \ \ \ \underset{r\to 0^+}{\lim} r^{2\alpha-1} \p_r f(r,t) = 0.
\end{equation}
If $\psi(y) = \Psi(|y|)$, we obtain from \eqref{gustavo}
\begin{align}\label{gustavo2}
u(x,t) & = (4\pi)^{- \frac m2} e^{m t  \omega} \left(\frac{2 \omega}{\sinh(2t \omega)}\right)^{\frac m2}
\times \int_0^\infty \Psi(\rho) \exp\left(- \frac{\omega}{2 \sinh(2t\omega)} [e^{2t\omega} \rho^2 + e^{-2 t \omega} |x|^2]\right)
\\
& \times  \int_{\mathbb S^{m-1}}  e^{ \frac{\omega \rho |x|}{\sinh(2t\omega)}\sa \frac{x}{|x|},y\da} d\sigma(y) \rho^{m-1} d\rho.
\notag
\end{align}
Using Bochner's argument (\cite[Theor. 40 on p. 69]{BC}), and the Poisson  representation of the modified Bessel function $I_\nu$, one has for $z>0$ and any $\xi\in \mathbb S^{m-1}$
\[
\int_{\mathbb S^{m-1}} \exp \left\{z\sa\xi,y\da\right\} d\sigma(y) =  (2\pi)^{\frac m2} z^{1-\frac m2} I_{\frac m2 - 1}(z),
\]
see e.g. p.2 in \cite{Gcontmat}.
Applying this formula with $z = \frac{\omega \rho |x|}{\sinh(2t\omega)}$, we find
\[
\int_{\mathbb S^{m-1}}  e^{ \frac{\omega \rho |x|}{\sinh(2t\omega)}\sa \frac{x}{|x|},y\da} d\sigma(y) = 
(2\pi)^{\frac m2} \left(\frac{\omega \rho |x|}{\sinh(2t\omega)}\right)^{1-\frac m2} I_{\frac m2 - 1}(\frac{\omega \rho |x|}{\sinh(2t\omega)}).
\]
Substituting this identity in \eqref{gustavo2}, we obtain
\begin{align}\label{gus3}
u(x,t) & =  e^{m t  \omega} \left(\frac{\omega}{\sinh(2t \omega)}\right)^{\frac m2} \left(\frac{\omega\rho |x|}{\sinh(2t\omega)}\right)^{1-\frac m2} I_{\frac m2 - 1}(\frac{\omega \rho|x|}{\sinh(2t\omega)})
\\
& \times \exp\left(- \frac{\omega e^{-2 t \omega} |x|^2}{2 \sinh(2t\omega)}\right) \int_0^\infty \Psi(\rho) \exp\left(- \frac{\omega e^{2t\omega} \rho^2}{2 \sinh(2t\omega)}\right)
  \rho^{m-1} d\rho.
\notag
\end{align}
If we now keep in mind that $m=2\alpha$, $r = |x|$, $\omega = \pi |\la|$, and $\Psi(\rho) = e^{\frac{\omega}{2} \rho^2} \hat \vf(\rho,\la)$, from \eqref{gus3} we obtain  for the solution of \eqref{cpou}
\begin{align}\label{gus4}
f(r,t) & =  e^{2\alpha t \pi |\la|} \left(\frac{\pi |\la|}{\sinh(2t \pi |\la|)}\right)^{\alpha} \left(\frac{\pi |\la| \rho r}{\sinh(2t\pi |\la|)}\right)^{1-\alpha} I_{\alpha - 1}(\frac{\pi |\la| \rho r}{\sinh(2t\pi |\la|)})
\\
& \times \exp\left(- \frac{\pi |\la| e^{-2 t \pi |\la|} r^2}{2 \sinh(2t\pi |\la|)}\right) \int_0^\infty e^{\frac{\pi |\la|}{2} \rho^2} \hat \vf(\rho,\la) \exp\left(- \frac{\pi |\la| e^{2t\pi |\la|} \rho^2}{2 \sinh(2t\pi |\la|)}\right)
  \rho^{2\alpha-1} d\rho
\notag\\
& =  e^{2\alpha t \pi |\la|}r^{1-\alpha} \frac{\pi |\la|}{\sinh(2t \pi |\la|)}   \exp\left(- \frac{2\pi |\la| e^{-2 t \pi |\la|} r^2}{4 \sinh(2t\pi |\la|)}\right) 
\notag
\\
& \times \int_0^\infty e^{\frac{\pi |\la|}{2} \rho^2} \hat \vf(\rho,\la) \exp\left(- \frac{2\pi |\la| e^{2t\pi |\la|} \rho^2}{4 \sinh(2t\pi |\la|)}\right) I_{\alpha - 1}(\frac{\pi |\la|\rho r}{\sinh(2t\pi |\la|)})
  \rho^{\alpha} d\rho.
\notag
\end{align}

Going back to \eqref{cpak}, keeping in mind that the connection between $v$ and $f$ is given by 
\[
f(r,t) = e^{\frac{\pi |\la|}{2} r^2 + 2\alpha \pi |\la| t} v(r,t), 
\]
and that $v(r,t) = \hat u((r,\la),t)$,
we obtain from \eqref{gus4}
\begin{align}\label{gus5}
\hat u((r,\la),t) & =   e^{-\frac{2\pi |\la|}{4} r^2} r^{1-\alpha} \frac{\pi |\la|}{\sinh(2t \pi |\la|)}  \exp\left(- \frac{2\pi |\la| e^{-2 t \pi |\la|} r^2}{4 \sinh(2t\pi |\la|)}\right) 
\\
& \times \int_0^\infty e^{\frac{2\pi |\la|}{4} \rho^2} \hat \vf(\rho,\la) \exp\left(- \frac{2\pi |\la| e^{2t\pi |\la|} \rho^2}{4 \sinh(2t\pi |\la|)}\right) I_{\alpha - 1}(\frac{\pi |\la| \rho r}{\sinh(2t\pi |\la|)}) 
  \rho^{\alpha} d\rho.
\notag
\end{align}
Using the identity
\[
1+ \frac{e^{-2t\pi |\la|}}{\sinh(2t\pi |\la|)} = \frac{1}{\tanh(2t\pi |\la|)},
\]
we can rewrite \eqref{gus5} in the following more compact form
\begin{align}\label{gus6}
\hat u((r,\la),t) & =   e^{-\frac{2\pi |\la|}{\tanh(2t \pi |\la|)} \frac{r^2}4} r^{1-\alpha} \frac{\pi |\la|}{\sinh(2t \pi |\la|)}   
\\
& \times \int_0^\infty \rho^{\alpha} e^{-\frac{2\pi |\la|}{\tanh(2t \pi |\la|)} \frac{\rho^2}4} I_{\alpha - 1}(\frac{\pi |\la| \rho r}{\sinh(2t\pi |\la|)})  \hat \vf(\rho,\la) d\rho
\notag\end{align}
Taking the inverse Fourier transform with respect to the variable $\la\in \R^k$ in \eqref{gus6}, we find
\begin{align}\label{gus7}
u((r,\sigma),t) & = r^{1-\alpha} \int_{\R^k} e^{2\pi i\sa \sigma,\la\da} e^{-\frac{2\pi |\la|}{\tanh(2t \pi |\la|)} \frac{r^2}4}  \frac{\pi |\la|}{\sinh(2t \pi |\la|)}  
\\
& \times \int_0^\infty \rho^{\alpha} e^{-\frac{2\pi |\la|}{\tanh(2t \pi |\la|)} \frac{\rho^2}4}I_{\alpha - 1}(\frac{\pi |\la|\rho r}{\sinh(2t\pi |\la|)}) 
 \hat \vf(\rho,\la)  d\rho d\la
\notag\\
& = r^{1-\alpha} \int_{\R^k} e^{2\pi i\sa \sigma,\la\da} e^{-\frac{2\pi |\la|}{\tanh(2t \pi |\la|)} \frac{r^2}4}  \frac{ \pi |\la|}{\sinh(2t \pi |\la|)}  
\notag
\\
& \times \int_0^\infty \rho^{\alpha} e^{-\frac{2\pi |\la|}{\tanh(2t \pi |\la|)} \frac{\rho^2}4} I_{\alpha - 1}(\frac{\pi |\la| \rho r}{\sinh(2t\pi |\la|)}) d\la\int_{\R^k} e^{-2\pi i\sa\la,\sigma'\da} \vf(\rho,\sigma') d\sigma' d\rho.
\notag
\end{align}
Reordering integrals, we obtain from \eqref{gus7}
\begin{align}\label{gus8}
u((r,\sigma),t) & = \int_0^\infty \int_{\R^k}\mathscr K_{\alpha,k}((r,\sigma),(\rho,\sigma'),t)  \vf(\rho,\sigma')   d\sigma' \rho^{2\alpha-1} d\rho
\end{align}
where
\begin{align*}\label{gus9}
\mathscr K_{\alpha,k}((r,\sigma),(\rho,\sigma'),t) & = (r\rho)^{1-\alpha}  \int_{\R^k} e^{-2\pi i\sa \sigma'-\sigma,\la\da} \frac{\pi |\la|}{\sinh(2t \pi |\la|)}
\\
& \times  e^{-\frac{2\pi |\la|}{\tanh(2t \pi |\la|)} \frac{r^2 +\rho^2}4} I_{\alpha - 1}(\frac{\pi |\la| \rho r}{\sinh(2t\pi |\la|)})\ d\la. 
\notag
\end{align*}
Setting $\la'=2\pi t\la$, so that $|\la'| = 2\pi t |\la|$, and $d\la' = (2\pi t)^k d\la$, we finally obtain \eqref{gus10}, thus completing the proof of Theorem \ref{T:gengus}.

\medskip

We next turn to the 
\begin{proof}[Proof of Proposition \ref{P:limit}]
Throughout, we fix $\alpha>0$, $k\in\mathbb N$, $r>0$, $t>0$ and $\sigma\in\mathbb R^k$.
If we define
\[
\Phi_{\rho,\sigma'}(\lambda) = \frac{(r\rho)^{1-\alpha}}{\pi^k (2t)^{k+1}}
e^{-\frac{i}{t}\langle \sigma'-\sigma,\lambda\rangle}
\frac{|\lambda|}{\sinh |\lambda|}
e^{-\frac{|\lambda|}{\tanh |\lambda|}
\frac{r^2+\rho^2}{4t}}
I_{\alpha-1}\!\left(
\frac{|\lambda|\rho r}{2t\sinh |\lambda|}
\right),
\]
then from \eqref{gus10} we have
\[
\mathscr K_{\alpha,k}((r,\sigma),(\rho,\sigma'),t)  = \int_{\R^k} \Phi_{\rho,\sigma'}(\lambda) d\la.
\]
If we let
\begin{equation}\label{s}
s_\rho(\lambda):= \frac{|\lambda|\rho r}{2t\sinh |\lambda|}, 
\end{equation}
then since $s_\rho(\lambda)\to0$ as $\rho\to0$, from the power series representation of $I_\nu$, see e.g. \cite[(5.7.1) on p. 108]{Le}
 \[
I_\nu(z) = \sum_{k=0}^\infty \frac{(z/2)^{\nu+2k}}{k! \G(\nu+k+1)},
\] 
we see that, with $\nu=\alpha-1> -1$, we have
\begin{align*}
& (r\rho)^{1-\alpha} I_{\alpha-1}(s_\rho(\lambda))
= \left(\frac{2t\sinh |\la|}{|\la|}\right)^{1-\alpha} s_\rho(\lambda)^{1-\alpha}I_{\alpha-1}(s_\rho(\lambda))
 \underset{\rho\to 0^+}{\longrightarrow}\  \frac{2^{1-\alpha}}{\Gamma(\alpha)}\left(\frac{2t \sinh |\la|}{|\la|}\right)^{1-\alpha}.
\end{align*}
This implies for each fixed $\lambda\in\mathbb R^k$ 
\begin{equation}\label{pointwise}
\lim_{(\rho,\sigma')\to(0,0)} \Phi_{\rho,\sigma'}(\lambda)
=\frac{1}{\Gamma(\alpha)(4t)^{\alpha-1}}e^{\,\frac{i}{t}\langle\sigma,\lambda\rangle}
\left(\frac{|\lambda|}{\sinh|\lambda|}\right)^{\alpha}
e^{-\frac{|\lambda|}{\tanh|\lambda|}\frac{r^2}{4t}} \frac{1}{\pi^k (2t)^{k+1}}.
\end{equation}
Therefore, if we can interchange the limit with the integral, we infer that
\[
\lim_{(\rho,\sigma')\to(0,0)} \mathscr K_{\alpha,k}((r,\sigma),(\rho,\sigma'),t)  = \frac{2\pi^\alpha}{\Gamma(\alpha)} \frac{2^{k}}{(4\pi t)^{\alpha+k}} \int_{\R^k} e^{\,\frac{i}{t}\langle\sigma,\lambda\rangle}
\left(\frac{|\lambda|}{\sinh|\lambda|}\right)^{\alpha}
e^{-\frac{|\lambda|}{\tanh|\lambda|}\frac{r^2}{4t}} d\la.
\]
Comparing with \eqref{gusrad}, we conclude that \eqref{limit} holds.

We now fix $\delta\le 1$, $\rho_0>0$ such that  
\begin{equation}\label{rho}
\frac{\rho_0 r}{2t} \le1,
\end{equation}
and define $Q_0 :=(0,\rho_0]\times B_{\delta}(0)$. To finish the proof, we only need to produce a dominating function $M(\lambda)\ge 0$ such that $\int_{\mathbb R^k} M(\lambda)\,d\lambda<\infty$, and for which for all $(\rho,\sigma')\in Q_0$, we have
\begin{equation}\label{dominate}
\left|\Phi_{\rho,\sigma'}(\lambda)\right|\le M(\lambda),\qquad\text{for a.e. }\lambda\in\mathbb R^k.
\end{equation}
Keeping in mnd that $\frac{s}{\sinh s}\le 1$ for $s\ge 0$, 
we have from \eqref{rho}
\[
0\le s_\rho(\lambda) \le \frac{\rho r}{2t} \le 1,\ \ \ \ \ \text{when}\ 0\le \rho\le \rho_0.
\]
Since by the power series representation of $I_\nu$ we have 
\[
0\le I_\nu(s)\le C_\nu s^\nu,\ \ \ \ \ 0\le s\le 1,
\]
we infer that for $0\le \rho\le \rho_0$
\[
I_{\alpha-1}\left(\frac{|\lambda|\rho r}{2t\sinh |\lambda|}\right) \le C_\alpha \left(\frac{|\lambda|\rho r}{2t\sinh |\lambda|}\right)^{\alpha-1}.
\]
Noting that $\frac{s}{\tanh s} \ge 1$ for every $s\ge 0$, we conclude that for every $(\rho,\sigma')\in Q_0$ and every $\la\in \R^k$, we have
\[
|\Phi_{\rho,\sigma'}(\lambda)| \le \frac{C_\alpha}{\pi^k (2t)^{k+\alpha}} e^{-
\frac{r^2}{4t}}
\left(\frac{|\lambda|}{\sinh |\lambda|}\right)^{\alpha} := M(\la)\in L^1(\R^k).
\]
This completes the proof.

\end{proof}

\section{Proof of Theorem \ref{T:gen}}\label{S:gen}

We say that a  function is spherically symmetric in $\R^k$ if $f(x) = f^\star(|x|)$, for some $f^\star$. When $k=1$, we simply intend that $f$ is even, and we let $f^\star(r) = f(r)$ for $\ge 0$. We begin by recalling the following classical formula due to Bochner, see \cite[Theor. 40 on p. 69]{BC}. 

\begin{proposition}\label{P:boch}
Assume that $k\ge 1$. Then, for any $\xi\in \R^k$ one has
\[
\hat f(\xi) = \frac{2\pi}{|\xi|^{\frac k2 - 1}} \int_0^\infty r^{\frac k2} f^\star(r) 
J_{\frac{k}2-1}(2\pi |\xi| r) dr.
\]
\end{proposition}

To see that the above formula continues to be valid when $k=1$, one needs to use the well-known identity 
\[
J_{-1/2}(z) = \sqrt{\frac{2}{\pi z}} \cos z,
\]
see \cite[(5.8.2) on p.111]{Le}. We then turn to the proof of Theorem \ref{T:gen}.

From \eqref{gen}, using Cavalieri's principle,  we find
\begin{align*}
& \mathscr E_{\alpha,\beta}(z,\sigma) = 2^k (4\pi)^{-\beta} \int_0^\infty \frac{1}{ t^{\beta-1}}\int_{\R^k} e^{-\frac it \langle\sigma,\la\rangle} \left(\frac{|\la|}{\sinh |\la|}\right)^{\alpha}  e^{- \frac{|z|^2}{4t} \frac{|\la|}{\tanh |\la|}}d\la\frac{dt}t
\\
& = 2^k (4\pi)^{-\beta} \int_0^\infty t^{\beta-2}\int_{\R^k} e^{-2\pi i  \langle\frac{t\sigma}{2\pi},\la\rangle} \left(\frac{|\la|}{\sinh |\la|}\right)^{\alpha}  e^{- \frac{t |z|^2}{4} \frac{|\la|}{\tanh |\la|}}d\la dt
\\
& = 2^k (4\pi)^{-\beta} \int_0^\infty t^{\beta-2} \hat f(\frac{t\sigma}{2\pi}) dt, 
\end{align*}
where we have let 
\begin{equation}\label{f}
f(\la) = \left(\frac{|\la|}{\sinh |\la|}\right)^{\alpha}  e^{- \frac{t |z|^2}{4} \frac{|\la|}{\tanh |\la|}}.
\end{equation}
Invoking Proposition \ref{P:boch}, we find
\[
\hat f(\frac{t\sigma}{2\pi}) = \frac{(2\pi)^{\frac k2}}{t^{\frac k2-1} |\sigma|^{\frac k2-1}} \int_0^\infty r^{\frac k2} \left(\frac{r}{\sinh r}\right)^{\alpha}  e^{- \frac{t |z|^2}{4} \frac{r}{\tanh r}} 
J_{\frac{k}2-1}(t |\sigma| r) dr.
\] 
Substituting this identity in the above expression of $\mathscr E_{\alpha,\beta}(z,\sigma)$, and exchanging the order of integration, we find
\begin{align}\label{E2}
& \mathscr E_{\alpha,\beta}(z,\sigma) =  \frac{2^k (4\pi)^{-\beta} (2\pi)^{\frac k2}}{|\sigma|^{\frac k2-1}} \int_0^\infty r^{\frac k2} \left(\frac{r}{\sinh r}\right)^{\alpha}  \int_0^\infty t^{\beta- \frac k2 -1} e^{- \frac{t |z|^2}{4} \frac{r}{\tanh r}} 
J_{\frac{k}2-1}(t |\sigma| r) dt dr.
\end{align}
We next recall the following classical formula due to Gegenbauer, see (3) on p. 385 in \cite{Wa}  
\begin{equation}\label{ohoh}
\int_0^\infty t^{\mu-1} e^{- a t} J_\nu(b t) dt = \frac{2^{-\nu} b^\nu \G(\nu+\mu)}{\G(\nu+1) (a^2 + b^2)^{\frac{\nu+\mu}2}} F\left(\frac{\nu+\mu}{2},\frac{1-\mu+\nu}{2};\nu+1;\frac{b^2}{a^2 + b^2}\right), 
\end{equation}
provided that
\[
\Re(\nu+\mu)>0,\ \Re(a+i b)>0,\ \Re(a-i b)>0.
\]
We apply \eqref{ohoh} with the choice
\[
\nu = \frac k2 -1,\ \ \ \mu = \beta - \frac k2,\ \ \ a = \frac{|z|^2}{4} \frac{r}{\tanh r},\ \ \ b = r |\sigma|.
\]
This gives
\[
\nu + \mu = \beta-1>0,\ \ \ \ 1-\mu+\nu = k-\beta.
\]
Notice that
\[
a^2 + b^2 = \frac{r^2}{16 \tanh^2 r} (|z|^4 + 16 |\sigma|^2 \tanh^2 r),\ \ \ \frac{\beta^2}{\alpha^2 + \beta^2} = \frac{16 |\sigma|^2 \tanh^2 r}{|z|^4 + 16 |\sigma|^2 \tanh^2 r}.
\]
We obtain
\begin{align*}
& \int_0^\infty t^{\beta- \frac k2 -1} e^{- \frac{t |z|^2}{4} \frac{r}{\tanh r}} 
J_{\frac{k}2-1}(t |\sigma| r) dt 
\\
& = \frac{2^{1-\frac k2} (r |\sigma|)^{\frac k2 -1} \G(\beta-1)}{\G(\frac k2) (|z|^4 + 16|\sigma|^2 \tanh^2 r)^{\frac{\beta-1}2}} \frac{(16 \tanh^2 r)^{\frac{\beta-1}2}}{r^{\beta-1}} F\left(\frac{\beta-1}{2},\frac{k-\beta}{2};\frac k2;\frac{16|\sigma|^2 \tanh^2 r}{|z|^4 + 16|\sigma|^2 \tanh^2 r}\right).
\end{align*}
Substituting this formula in \eqref{E2}, and at this point using the crucial assumption $\boxed{\beta = \alpha+k}$, we see that the powers of $r$ cancel, and we find
\begin{align}\label{E3}
& \mathscr E_{\alpha,\beta}(z,\sigma) =  \frac{2^{1+\frac k2} (4\pi)^{-\beta} (2\pi)^{\frac k2}\G(\beta-1)4^{\beta-1}}{\G(\frac k2)} \int_0^\infty r^{k+\alpha-\beta} \left(\frac{1}{\sinh r}\right)^{\alpha} (\tanh r)^{\beta-1}
\\
& \times \frac{1}{ (|z|^4 + 16|\sigma|^2 \tanh^2 r)^{\frac{\beta-1}2}} F\left(\frac{\beta-1}{2},\frac{k-\beta}{2};\frac k2;\frac{16|\sigma|^2 \tanh^2 r}{|z|^4 + 16|\sigma|^2 \tanh^2 r}\right) dr
\notag
\\
& = \frac{2^{k-1} \pi^{\frac k2-\beta}\G(\beta-1)}{\G(\frac k2)} \int_0^\infty \left(\frac{1}{\sinh r}\right)^{\alpha} (\tanh r)^{\alpha}(\tanh r)^{\beta-\alpha-1}
\notag
\\
& \times \frac{1}{ (|z|^4 + 16|\sigma|^2 \tanh^2 r)^{\frac{\beta-1}2}} F\left(\frac{\beta-1}{2},\frac{k-\beta}{2};\frac k2;\frac{16|\sigma|^2 \tanh^2 r}{|z|^4 + 16|\sigma|^2 \tanh^2 r}\right) dr
\notag
\\
& = \frac{2^{k-2} \pi^{\frac k2-\beta}\G(\beta-1)}{\G(\frac k2)} \int_0^\infty \left(\frac{1}{\cosh^2 r}\right)^{\frac{\alpha}2 -1} (\tanh^2 r)^{\frac{\beta-\alpha}{2}-1} \frac{1}{ (|z|^4 + 16|\sigma|^2 \tanh^2 r)^{\frac{\beta-1}2}}
\notag
\\
& \times  F\left(\frac{\beta-1}{2},\frac{k-\beta}{2};\frac k2;\frac{16|\sigma|^2 \tanh^2 r}{|z|^4 + 16|\sigma|^2 \tanh^2 r}\right) \frac{2 \tanh r}{\cosh^2 r} dr
\notag
\\
& = \frac{2^{k-2} \pi^{\frac k2-\beta}\G(\beta-1)}{\G(\frac k2)} \int_0^\infty \left(1-\tanh^2 r\right)^{\frac{\alpha}2 -1} (\tanh^2 r)^{\frac{\beta-\alpha}{2}-1} \frac{1}{ (|z|^4 + 16|\sigma|^2 \tanh^2 r)^{\frac{\beta-1}2}}
\notag
\\
& \times  F\left(\frac{\beta-1}{2},\frac{k-\beta}{2};\frac k2;\frac{16|\sigma|^2 \tanh^2 r}{|z|^4 + 16|\sigma|^2 \tanh^2 r}\right) \frac{2 \tanh r}{\cosh^2 r} dr.
\notag
\end{align}
In the integral in the right-hand side of \eqref{E3} we now make the change of variable $y = \tanh^2 r$, for which $dy = \frac{2\tanh r}{\cosh^2 r} dr$, obtaining
\begin{align}\label{E3}
& \mathscr E_{\alpha,\beta}(z,\sigma) = \frac{2^{k-2} \pi^{\frac k2-\beta}\G(\beta-1)}{\G(\frac k2)} \int_0^1 \left(1-y\right)^{\frac{\alpha}2 -1} y^{\frac{\beta-\alpha}{2}-1} \frac{1}{ (|z|^4 + 16|\sigma|^2 y)^{\frac{\beta-1}2}}
\\
& \times  F\left(\frac{\beta-1}{2},\frac{k-\beta}{2};\frac k2;\frac{16|\sigma|^2 y}{|z|^4 + 16|\sigma|^2 y}\right)  dy.
\notag
\end{align}
Next, we recall the following Kummer's relation concerning how the hypergeometric function $F$ changes under linear transformations (see  formula (3) on p. 105 in \cite{E}, or also (9.5.1) on p. 247 in \cite{Le}),
\begin{equation}\label{hyperG}
F(a,b;c;u) = (1-u)^{-a} F\left(a,c - b;c;\frac{u}{u-1}\right),\ \ \ \ \ \ u\not=1, \ |\arg(1-u)|<\pi, 
\end{equation}
valid for any $a, b,c\in \mathbb C$, with $c\not= 0, - 1, - 2,...$ We use \eqref{hyperG} with the choices
\[
a = \frac{\beta-1}2,\ \ \ c = \frac k2,\ \ \ c - b = \frac{k-\beta}{2},\ \ \ \frac{u}{u-1} = \frac{16 |\sigma|^2 y}{|z|^4 + 16 |\sigma|^2 y}.
\]
Notice that, with these choices, we have 
\[
b = \frac \beta{2},
\]
and that 
\[
u = - \frac{16|\sigma|^2}{|z|^4} y,\ \ \ \ \ 1-u = \frac{|z|^4 + 16 |\sigma|^2 y}{|z|^4}.
\] 
We thus find from \eqref{hyperG}
\begin{align}\label{holyminchias}
& F\left(\frac{\beta-1}{2},\frac{k-\beta}{2};\frac k2;\frac{16|\sigma|^2 y}{|z|^4 + 16|\sigma|^2 y}\right)  
 = \frac{(|z|^4 + 16 |\sigma|^2 y)^{\frac{\beta-1}2}}{|z|^{2(\beta-1)}}  F\left(\frac{\beta-1}{2},\frac{\beta}2;\frac k2;- \frac{16|\sigma|^2}{|z|^4} y\right).
\end{align}
We now use \eqref{holyminchias} in \eqref{E3}, obtaining
\begin{align}\label{E4}
& \mathscr E_{\alpha,\beta}(z,\sigma) = \frac{2^{k-2} \pi^{\frac k2-\beta}\G(\beta-1)}{\G(\frac k2)|z|^{2(\beta-1)}} \int_0^1 y^{\frac{\beta-\alpha}{2}-1} \left(1-y\right)^{\frac{\alpha}2 -1}  F\left(\frac{\beta-1}{2},\frac{\beta}2;\frac k2;- \frac{16|\sigma|^2}{|z|^4} y\right) dy
\end{align}
We now use the following formula due to H. Bateman (see formula (2) on p. 78 in \cite{E}, and also Problem 6. on p. 277 in \cite{Le})
\begin{equation}\label{hyperGint}
\int_0^1 y^{c-1} (1-y)^{\gamma-c-1} F(a,b;c;\delta y) dy = \frac{\G(c)\G(\gamma - c)}{\G(\gamma)} F(a,b;\gamma;\delta),
\end{equation}
provided that $\Re \gamma >\Re c>0$, $\delta\not=1$ and $|\arg(1-\delta)|<\pi$. To apply \eqref{hyperGint} we must choose
\begin{equation}\label{abc}
a = \frac{\beta-1}2,\ \ \ b = \frac{\beta}2,\ \ \ c = \frac k2>0,\ \ \ \delta = - \frac{16|\sigma|^2}{|z|^4}.
\end{equation}
Moreover, we must make sure that $c = \frac k2$, but this is true since $\beta - \alpha = k$ and comparing \eqref{hyperGint} with \eqref{E4} we see that $c = \frac{\beta-\alpha}{2}$. The same comparison shows that we must have $\gamma-c = \frac{\alpha}2$, which gives
\[
\gamma = c + \frac{\alpha}2 = \frac{\alpha+k}2 = \frac{\beta}2.
\] 
Applying \eqref{hyperGint} we thus find
\[
\int_0^1 y^{\frac{\beta-\alpha}{2}-1} \left(1-y\right)^{\frac{\alpha}2 -1}  F\left(\frac{\beta-1}{2},\frac{\beta}2;\frac k2;- \frac{16|\sigma|^2}{|z|^4} y\right) dy = \frac{\G(\frac k2)\G(\frac{\alpha}2)}{\G(\frac{\beta}2)} F(\frac{\beta-1}2,\frac{\beta}2;\frac{\beta}2;- \frac{16|\sigma|^2}{|z|^4}).
\]
We see now that a miracle has happened, in the sense that the hypergeometric function is in the special form
\begin{equation}\label{F}
F(a,b;b;-\delta) =\ _1F_0(a;-\delta) = (1+\delta)^{-a},
\end{equation}
see formula (4) on p. 101 in \cite{E}. We thus find
\begin{align*}
& F(\frac{\beta-1}2,\frac{\beta}2;\frac{\beta}2;- \frac{16|\sigma|^2}{|z|^4}) = \left(1+\frac{16|\sigma|^2}{|z|^4}\right)^{-\frac{\beta-1}2}
= \frac{|z|^{2(\beta-1)}}{\left(|z|^4 + 16 |\sigma|^2\right)^{2(\beta-1)}}.
\end{align*}
Substituting in the above, we find
\begin{align*}\label{E5}
& \mathscr E_{\alpha,\beta}(z,\sigma) = \frac{2^{k-2} \pi^{\frac k2-\beta}\G(\frac{\alpha}2)\G(\beta-1)}{\G(\frac{\beta}2)} \frac{1}{\left(|z|^4 + 16 |\sigma|^2\right)^{\frac{2(\beta-1)}{4}}}.
\end{align*} 
Using Legendre's duplication formula for the gamma function we have
\[
\frac{\G(\beta-1)}{\G(\frac{\beta}2)} = \frac{2^{\beta-1} \G(\frac{\beta+1}2)}{\sqrt \pi (\beta-1)}.
\]
Substituting in the previous formula, and using \eqref{N}, we finally obtain 
\begin{align*}\label{E5}
& \mathscr E_{\alpha,\beta}(z,\sigma) = \frac{2^{k-2} \pi^{\frac k2-\beta}\G(\frac{\alpha}2)\G(\beta-1)}{\G(\frac{\beta}2)} \frac{1}{\left(|z|^4 + 16 |\sigma|^2\right)^{\frac{2(\beta-1)}{4}}}.
\end{align*} 
Finally, keeping in mind that $\beta = \alpha + k$, and using again Legendre's formula, which gives
\[
\frac{\G(\beta-1)}{\G(\frac \beta2)} = \frac{\G(\alpha+k-1)}{\G(\frac{\alpha+k}2)} = \pi^{-1/2} 2^{\alpha+k-2} \G(\frac{\alpha+k-1}2),
\]
we finally obtain 
\begin{align*}
& \mathscr E_{\alpha,\alpha+k}(z,\sigma) = \frac{2^{\alpha + 2k-4}\G(\frac{\alpha}2) \G(\frac{\alpha + k -1}2)}{\pi^{\frac{2\alpha+k+1}2}} N(z,\sigma)^{-(2\alpha+2k-2)}.
\end{align*}
This establishes \eqref{gen}, thus completing the proof of Theorem \ref{T:gen}.

\section{Connections with nonlinear equations}\label{S:np}

In this final section we point out an intriguing connection between the above linear results and a nonlinear geometric flow that presently remains  unexplored. The discussion that follows is intended only as a preliminary indication of possible directions and its main motivation is to lead to the open Problems \ref{P:one} and \ref{P:two} below.
Consider in $\Rn$ the $p$-energy functional
\[
J_p(u) = \frac1p\int_{\Rn} |\nabla u|^p\, dx,
\]
whose Euler–Lagrange equation is the nonlinear $p$-Laplace equation
\begin{equation}\label{p}
\Delta_p u = \operatorname{div}(|\nabla u|^{p-2}\nabla u)=0.
\end{equation}
Its fundamental solution is the well-known power law
\begin{equation}\label{ep}
e_p(x)= C_{n,p}\, |x|^{-\frac{n-p}{p-1}},
\end{equation}
which reduces to the classical Newtonian potential when $p=2$.

The parabolic version $\partial_t u=\Delta_p u$ of \eqref{p} is strongly asymmetric in space and time, and  does not correspond to a geometric flow. A more appropriate evolution is the normalized $p$-Laplacian,
\[
\partial_t u = |\nabla u|^{2-p}\Delta_p u.
\]
At points of non-degeneracy of the gradient such equation can be equivalently written
\begin{equation}\label{np}
\partial_t u
   = \Delta u + (p-2)|\nabla u|^{-2}\Delta_\infty u,
\end{equation}
which connects in the limit $p\to 1$ to the mean curvature flow studied in the celebrated works \cite{ES1}, \cite{ES2}, \cite{ES3}, \cite{ES4}, \cite{SZ, CGG, CW, GGIS, ISZ}. The equation \eqref{np} enjoys a rich theory of viscosity solutions \cite{BG, Do} and reveals deep connections with stochastic tug-of-war games \cite{PS, MPR}. The a priori estimates in \cite{BG2} and the breakthrough regularity theory in \cite{JS} further highlight its geometric nature. In connection with Riemannian geometry, we mention the work in preparation \cite{BGM} where a Li-Yau theory and an optimal Harnack inequality have been developed for the flow \eqref{np} in the Riemannian case.

One remarkable aspect of the fully nonlinear operator \eqref{np}, which also constitutes its main link to the present work, is that on functions which are spherically symmetric in the space variable, it  acts as a linear operator in a fractal dimension. The key observation is that when $u(x,t)=f(r,t)$, $r=|x|$, then
\begin{equation}\label{fully}
|\nabla u|^{-2}\Delta_\infty u = f_{rr},
\end{equation}
so that, on such functions, the fully nonlinear equation \eqref{np} reduces to the linear
\begin{equation}\label{f}
f_t = (p-1)f_{rr} + \frac{n-1}{r}f_r.
\end{equation}
Rescaling time, $f(r,t) = g(r,(p-1)t)$, shows that $f$ is governed by the heat flow in a space $\R^\kappa\times (0,\infty)$ of ``fractal dimension''
\begin{equation}\label{k}
\kappa = \frac{n+p-2}{p-1}.
\end{equation}
This led the authors of \cite{BG} to discover the following notable solution of \eqref{np}:
\begin{equation}\label{gp}
g_p(x,t)
 = t^{-\frac{n+p-2}{2(p-1)}}
   \exp\bigg(-\frac{|x|^2}{4(p-1)t}\bigg).
\end{equation}
We stress that \eqref{gp} satisfies
\[
\int_{\Rn} g_p(x,t)\, dx
   = (4\pi(p-1))^{\frac n2}
     t^{\frac{(n-1)(p-2)}{2(p-1)}},
\]
and therefore it is not a fundamental solution, except when $p=2$ or $n=1$. Despite this, one still has the conformal property
\begin{equation}\label{intgp}
\int_0^\infty g_p(x,t)\, dt
   = c_{n,p}\, |x|^{-\frac{n-p}{p-1}},
\end{equation}
which mirrors the linear case $p=2$ in view of \eqref{ep}.

These considerations led us to ask whether the intertwining between linear and nonlinear structures observed in Euclidean space extends to the sub-Riemannian setting of the present study. For instance, in view of the cited works \cite{BG, BGM}, it would be quite interesting to know the answer to the following question.

\begin{prob}\label{P:one}
Does a counterpart $G_p((z,\sigma),t)$ of the prototypical solution \eqref{gp} exists in a group of Heisenberg type?
\end{prob}

Answering this question is a nontrivial task since in sub-Riemannian geometry there is no counterpart of the phenomenon \eqref{fully}. We show below that Theorems \ref{T:gengus} and \ref{T:gen} offer a possible candidate.

We recall that, in these Lie groups, the foundational work of Kor\'anyi-Reimann \cite{KoR, KoRaim}, Mostow \cite{Mo} and Margulis-Mostow \cite{MM} have underscored the role  of nonlinear equations such as
\begin{equation}\label{pH}
\Delta_{H,p} f = \operatorname{div}_H(|\nabla_H f|^{p-2}\nabla_H f)=0.
\end{equation}
A basic result (see \cite[Theor. 2.1]{CDGcap} and independently \cite{HH} for $p=Q$) shows that the Korányi gauge continues to govern nonlinear potential theory:
\begin{equation}\label{fs}
E_p(z,\sigma)
 = \begin{cases}
   C_p\, N(z,\sigma)^{-\frac{Q-p}{p-1}}, & p\neq Q, \\[6pt]
   C_p\, \log N(z,\sigma), & p=Q.
   \end{cases}
\end{equation}

\vskip 0.2in

\begin{prob}\label{P:two}
 In the Heisenberg group $\Hn$ we propose to study the following counterpart of the geometric flow \eqref{np}
\begin{equation}\label{fullyH}
\p_t u =  \Delta_H u + (p-2)|\nabla_H u|^{-2}\Delta_{H,\infty} u,\ \ \ \ \ \ 1<p<Q = 2n+2,
\end{equation}
where we have let $\Delta_{H,\infty} u = \frac 12 \sa\nh u,\nh(|\nh u|^2)\da$. More specifically, we propose to show existence and uniqueness of viscosity solutions of the Cauchy problem for \eqref{fullyH}, as well as the optimal regularity of their horizontal gradient. 
\end{prob}

\vskip 0.2in

Formally, \eqref{fullyH} converges as $p\to 1$ to the horizontal mean curvature flow studied in \cite{CC}, but the situation is now much more complex than its Euclidean  counterpart. The analysis of \eqref{fullyH} is presently lacking and it would lead to various interesting developments.

In connection with Problem \ref{P:one}, and motivated by the Euclidean formula \eqref{gp} and Theorem \ref{T:gen}, in the Heisenberg group $\Hn$ we introduce the following Fourier integral
\begin{align}\label{Gnew}
G_p((z,\sigma),t)
  &=
    t^{-\frac{Q+p-2}{2(p-1)}}
    \int_{\R}
    e^{-\frac{i}{(p-1)t}\langle\sigma,\lambda\rangle}
    \left(\frac{|\lambda|}{\sinh |\lambda|}\right)^{\frac{Q-p}{2(p-1)} }
    e^{-\frac{|z|^2}{4(p-1)t}\frac{|\lambda|}{\tanh |\lambda|}}
    \, d\lambda,\ \ \ 1<p<Q.
\end{align} 
We emphasize that, when $p=2$, the function \eqref{Gnew} reduces to a constant multiple of the classical Gaveau-Hulanicki kernel \eqref{GH}. Concerning \eqref{Gnew} we have the following consequence of our Theorem \ref{T:gen}.

\begin{proposition}\label{P:Gp}
For any $n\in \mathbb N$ and $1<p<Q = 2n+2$, there exists an explicit universal constant $C_{n,p}>0$ such that
\[
\int_0^\infty G_p((z,\sigma),t)\, dt
   =  C_{n,p}\, N(z,\sigma)^{-\frac{Q-p}{p-1}}.
\]
\end{proposition}

\begin{proof}
The proof of Proposition \ref{P:Gp} follows by taking $\alpha = \frac{Q-p}{2(p-1)}$ in Theorem \ref{T:gen}, and noting that, since for the Heisenberg group $\Hn$ one has $k=1$, we presently have
\[
\beta = \alpha + k = \frac{Q-p}{2(p-1)} + 1 = \frac{Q+p-2}{2(p-1)}.
\]
This observation shows that, up to the time scaling $t\to (p-1)t$, the function \eqref{Gnew} is exactly a $G_{\alpha,\alpha+k}$ as in Theorem \ref{T:gen}, with $k=1$ and $\alpha$ as above.

\end{proof}
Thus in $\Hn$ the nonlinear potential \eqref{fs} arises from the time integral of the kernel \eqref{Gnew}, just as in the Euclidean identity \eqref{intgp}. This parallel strongly suggests a deeper relationship between the kernel $G_p$ and the fully nonlinear operator \eqref{pH}, and Proposition \ref{P:Gp} seems to indicate in the function \eqref{Gnew} the appropriate sub-Riemannian analogue of the prototypical solution \eqref{gp}. However, the actual verification that $G_p$ solves \eqref{fullyH} involves complex computations and it is presently unresolved.

We mention in closing that, for both Problems \ref{P:one} and \ref{P:two}, we have some interesting work in progress.


\bibliographystyle{amsplain}

\end{document}